\documentclass[oneside,english,fullpage]{amsart}
\usepackage[T1]{fontenc}
\usepackage[latin9]{inputenc}
\usepackage{float}
\usepackage{amsthm}
\usepackage{graphicx}

\makeatletter

\providecommand{\tabularnewline}{\\}

\numberwithin{equation}{section}
\numberwithin{figure}{section}
\theoremstyle{plain}
\newtheorem{thm}{\protect\theoremname}
  \theoremstyle{definition}
  \newtheorem{defn}[thm]{\protect\definitionname}
  \theoremstyle{plain}
  \newtheorem{lem}[thm]{\protect\lemmaname}
  \theoremstyle{plain}
  \newtheorem{conjecture}[thm]{\protect\conjecturename}

\usepackage{geometry}
\makeatother

\usepackage{babel}
  \providecommand{\conjecturename}{Conjecture}
  \providecommand{\definitionname}{Definition}
  \providecommand{\lemmaname}{Lemma}
\providecommand{\theoremname}{Theorem}

\begin{document}

\title{Eigenvalue Dependence of Numerical Oscillations in Parabolic Partial
Differential Equations}

\author{Corban Harwood}

\author{Mitch Main}
\begin{abstract}
This paper investigates oscillation-free stability conditions of numerical
methods for linear parabolic partial differential equations with some
example extrapolations to nonlinear equations. Not clearly understood,
numerical oscillations can create infeasible results. Since oscillation-free
behavior is not ensured by stability conditions, a more precise condition
would be useful for accurate solutions. Using Von Neumann and spectral
analyses, we find and explore oscillation-free conditions for several
finite difference schemes. Further relationships between oscillatory
behavior and eigenvalues is supported with numerical evidence and
proof. Also, evidence suggests that the oscillation-free stability
condition for a consistent linearization may be sufficient to provide
oscillation-free stability of the nonlinear solution. These conditions
are verified numerically for several example problems by visually
comparing the analytical conditions to the behavior of the numerical
solution for a wide range of mesh sizes.
\end{abstract}



\maketitle
\global\long\def\dx{\Delta x}
\global\long\def\dt{\Delta t}

\section{Introduction}

Numerical methods are useful for constructing quick estimates of solutions
to partial differential equations (PDEs). Error can, however, creep
into the solution making results inaccurate and, at times, physically
meaningless due to stability issues. Yet, ensuring the stability of
the numerical solution to a well-posed PDE ensures convergence to
the true solution of the PDE. Von Neumann \cite{Von Neumann} provided
an analytical method for determining the stability of linear numerical
schemes, which he refined from the stability analysis provided first
by Crank and Nicolson \cite{Crank and Nicolson}. Numerical oscillations
can also occur, either as a small noise in the solution or as dramatic
swings in the solution leading to an instability. Scientists have
long been concerned with these oscillations, but have most often responded
by damping all oscillations. For example, Britz et al. \cite{Britz et al}
developed a damping algorithm specifically for oscillations caused
by discontinuous initial conditions for the Crank-Nicolson scheme
based off the initial work by �sterby \cite{Osterby}. But simply
damping these oscillations can change the nature of the solution itself,
especially if the solution was naturally oscillatory to begin with.
Rather, if we understand the nature of these oscillations and can
identify why they occur, then we can avoid them altogether. 

Previous work by Harwood \cite{Harwood} suggests there exists a more
restrictive condition for oscillation-free conditions as noted by
the eigenvalues of the time-step matrix. A spectral analysis would
be appropriate to determine the appropriate conditions on the scheme,
providing a necessary and sufficient condition for oscillation-free
solutions. 

Nonlinear problems provide a different challenge, but a consistent
linearization technique could provide insight through spectral or
Von Neumann analyses. Such PDEs have a range of application.  Each
of these models are prone to oscillation given certain conditions.
This paper investigates the relationship between eigenvalues and oscillatory
behavior and proposes necessary and sufficient conditions for oscillation-free
stable numerical solutions.

\section{Theory}

\subsection{Numerical Convergence}

Convergence of a numerical solution to the exact solution of a given
PDE is paramount to implementing a numerical method, but this form
of convergence is hard to measure. Instead, a numerical method can
be proven convergent indirectly by showing it to be consistent and
stable. A numerical method is \emph{consistent} if it converges to
the given PDE as the step size all go to zero. Further, a method is
\emph{stable} if any errors are bound instead of being magnified. 
\begin{defn}
\textit{A two-level difference scheme, $U^{n+1}=MU^{n},$ is said
to be stable with respect to the norm $||.||$ if there exist positive
constants $\triangle t_{0},\ and\ \triangle x_{0},$ and non-negative
constants K and $\beta$ so that 
\[
||U^{n+1}||\leq Ke^{\beta\triangle t}||U_{0}||,
\]
for $0\leq t=(n+1)\triangle t,\ 0<\triangle x\leq\triangle x_{0}\ and\ 0<\triangle t\leq\triangle t_{0}$}.
\label{Burden Faires Stability}
\end{defn}
Since $||U^{n+1}||\leq||M^{n+1}||\cdot||U_{0}||$, $||M^{n+1}||\leq Ke^{\beta\dt}$
is sufficient for stability of the solution. To compensate for $\rho(M^{n+1})\leq||M^{n+1}||_{2}$
we force the Von Neumann criterion
\begin{equation}
\rho(M)\leq1+C\triangle t,\label{VN criterion}
\end{equation}
for some $C\geq0$, so that $\rho(M^{n+1})\leq1+C(n+1)\triangle t\leq Ke^{C(n+1)\triangle t}$,
which is a necessary condition for stability. If the matrix is symmetric,
$\rho(M^{n+1})=||M^{n+1}||\leq1+C(n+1)\triangle t\leq e^{C(n+1)\triangle t}$,
then the criterion is necessary and sufficient. Often $C=0$ is adequate
to ensure a necessary and sufficient condition for stability, but
with exponentially growing solutions, the Von Neumann criterion may
require $C>0$. Solving for the amplification factor, $e^{\gamma\dt}=\frac{U_{m}^{n+1}}{U_{m}^{n}}$.
The condition is applied such that $|e^{\gamma\triangle t}|\leq1+C\triangle t$.
However, if the eigenvalues are easily calculated, we can determine
more accurately the stability and behavior of the solution.
\begin{thm}
Lax-Richtemeyer Theorem: \label{thm:Lax-Richtemeyer-Theorem}

``Given a Partial Differential Equation and a consistent numerical
method, a numerical method is convergent iff it is stable.'' 
\end{thm}
And more specifically, for two-level numerical method, this can be
restated as the Lax Equivalence theorem : 
\begin{thm}
Lax Equivalence Theorem: 

``A consistent, two-level difference scheme, $u^{n+1}=MU^{n}+\triangle t\cdot G^{n}$,
for a well-posed linear (constant-coefficient) initial-value problem
is convergent if and only if it is stable.'' \label{thm:Lax-Equivalence-Theorem:}
\end{thm}
Using consistent schemes, we can explore the stability bounds to indirectly
verify convergence. Burden and Faires \cite{Burden and Faires} define
stability for a numerical system as:

\subsection{Spectral Analysis}

For linear PDEs, the eigenvalues of the transformation matrix indicate
the stability of the solution and evidence suggests they also determine
oscillatory behavior. Since $\rho(M)\leq||M||_{2}$, an analysis of
the eigenvalues relates to the boundedness of the solution through
the matrix \cite{Horn and Johnson}. Assuming that the solution to
difference scheme is separable, we can show that the discrete error
growth factors are the eigenvalues of M. Also, assuming the solution
is separable, as is common for linear PDEs, the eigenvalues can be
defined by: 
\begin{equation}
e^{\gamma\dt}\equiv\frac{\epsilon_{m}^{n+1}}{\epsilon_{m}^{n}}=\frac{U_{m}^{n+1}}{U_{m}^{n}}=\frac{T^{n+1}}{T^{n}}=\lambda,\ \mbox{where }U_{m}^{n}=X_{m}T^{n}.\label{Growth Factor-Eigenvalue relation}
\end{equation}
 This shows that when separation of variables can be assumed in the
PDE, the eigenvalues of the matrix for the numerical scheme equals
the error growth factor. This relationship is supported by the similarity
of Von Neumann stability analysis to the Matrix Convergence Theorem
proven in \cite{Friedberg et Al}:
\begin{thm}
\textbf{\emph{(Matrix Convergence Theorem)}}$\lim_{n\rightarrow\infty}M^{n}$
exists iff $|\lambda_{i}|<1$ or $\lambda_{i}=1,\ \forall i$ and
the multiplicity of $\lambda=1$ is equal to the dimension of the
eigenspace, $E_{1}$. \label{Matrix Convergence Theorem}
\end{thm}
Furthermore, $lim_{n\rightarrow\infty}M^{n}$ is bounded iff $|\lambda|\leq1$
when $mult(1)=dim(E_{1})$ \cite{Friedberg et Al}. Because of this
connection, a spectral analysis of the transformation matrix can reveal
patterns in the orientation, spread, and balance of the growth of
errors for various wave modes and unbalanced eigenvalues can lead
to oscillations in the results. 

These troublesome oscillations have several sources. Unstable oscillations
stem from unstable methods, where the algorithm is either inconsistent
or the discretization is too large. Other sources are discontinuities:
in the initial condition, or between the initial condition and boundary
condition. The Crank-Nicolson method can be susceptible to these discontinuities,
which propagate oscillations (stable and unstable) in a normally consistent
and stable method. The nature of the PDE itself can create oscillations,
but these are generally more acceptable since these are expected and
are part of the solution. However, we are concerned with stable oscillations.
Theoretically, oscillations are hard to define. From a physical study
of springs, one can define oscillations as a type of periodic movement
or displacement from an equilibrium position \cite{Peterson and Sochaski}.
While this is visually descriptive, it is lacking analytically. Mathematically
defining oscillations is difficult, but numerically describing them
can be easier. Since oscillations displace positively and negatively
from the equilibrium, we implemented a working numerical definition
for oscillation. By tracking the sign of the displacement, we could
track if the signs of the spatial components flipped in time. Pearson
\cite{Pearson} and Britz et al. \cite{Britz et al} stated that preserved
monotonicity in time could suggest oscillation-free behavior for a
solution. While not an exact definition of oscillation, monotonicity
provides a sufficient condition for oscillation-free stability. Thus,
we used a theoretical definition for temporal monotonicity: $U_{m}^{n+1}\leq U_{m}^{n}$
or $U_{m}^{n+1}\geq U_{m}^{n}$ \label{Monotonicity Definition}.
Spatial monotonicity is defined in a similar manner where the lower
index changes instead. 

With this information, if a region of stable oscillations were to
exist for a particular scheme, then it would have to be bounded above
by the stability condition (obviously, since we wish to account for
oscillations inside the stability bounds) and bounded below by the
monotonicity condition.

Tridiagonal matrices create a very interesting special case. Each
equation we investigated, under Dirichlet boundaries, created a tridiagonal
matrix of a form whose eigenvalues were calculated by \cite{Harwood}.
Furthermore we could reasonably restrict the eigenvalues further to
an oscillation free condition, by forcing the real parts of the eigenvalues
positive \cite{Harwood}. For notation, we will refer to the upper
off diagonal entries as $a$, the diagonal as $b$ and the lower off
diagonal as $c$. For such tridiagonal matrices, the eigenvectors
with first component $u_{i}^{(1)}=\sin\left(\frac{i\pi}{N+1}\right),\ \forall i$,
are defined by their $k^{th}$ component as:
\[
u_{i}^{(k)}=\sin\left(\frac{ik\pi}{N+1}\right),\ \forall i
\]

This method is effective in a spectral analysis of Dirichlet boundary
conditions. The explicit Euler Heat equation is a prime example for
tridiagonal matrices. In the explicit Heat Equation $a=c$, thus the
eigenvalues and eigenvectors are independent of a square root and
the equation is quickly simplified. Also note that in the Heat FTCS
scheme, $b=1-2r$, where $r=\frac{\triangle t}{\triangle x^{2}}$.
This means that the upper bound on the eigenvalues is 1, which is
the upper bound to stability. We can calculate these eigenvalues with:

\[
\lambda_{i}=1-4r\sin^{2}\left(\frac{i\pi}{2(N+1)}\right)
\]
where $N$ is the size of the matrix.

By calculating the max and min we can definitively say that $\lambda_{i}\in(1-4r,1)$
but as $r$ increases in size, the eigenvalues leak into the negative
real parts causing oscillations. Thus we can use eigenvalues to more
restrictively bound oscillation-free stability. Other boundaries can
change the nature of the matrix and make calculating the eigenvalues
difficult. Neumann problems rid us of our tridiagonal assumption,
making eigenvalues more difficult to analyze. In these cases, since
we can assume a linear separable solution, the Von Neumann error factor
can be used to determine the behavior of the eigenvalues. In analyzing
the Heat Crank-Nicolson scheme matrix, a spectral analysis of the
eigenvalues showed that Dirichlet problems for the Heat equation lie
in the set $\lambda_{i}\in\left[\frac{1-2r}{1+2r},1\right)$. However
the Neumann problem showed that the largest eigenvalue was always
1. This meant that Dirichlet problems could experience oscillations
if the eigenvalues were to become negative, but all Neumann problems
were going to be stable, regardless of the step sizes used for discretization.

\subsection{Linearized Analysis of Nonlinear Problems}

For nonlinear problems, we wished to investigate three main methods
of calculating a numerical solution. An explicit scheme which doesn't
require any linearization but is very susceptible to numerical oscillations
as well as instability due to the nonlinearities. A semi-implicit
scheme allows us to explicitly solve for the linear portions of the
equation, but tackle the nonlinearities implicitly. This creates problems
computationally since the matrix requires updates in order to implicitly
solve the nonlinearities, making it very costly to constantly update
and resume solving. Instead, an implicit scheme which employs a linearization
makes computation easier. For one technique, we chose a linear approximation
of the nonlinear term: 
\begin{equation}
f(U_{m}^{n+1})\approx f(U_{m}^{n})+\triangle U\cdot\frac{\partial f(U^{n})}{\partial u}+....\approx f(U^{n})+(U^{n+1}-U^{n})\cdot f_{u}(U^{n})+...\label{Linear Approximation}
\end{equation}
This technique aims to approximate the nonlinear pieces independently,
similar to the manner used for the finite differences schemes (both
are based off of Taylor Series). Because this requires no updates
and runs alongside the normal computation, we can solve this linearization
scheme relatively easily. Our second technique is a nonlinear freezing
technique, where we separate off a linear term from the nonlinear
and treat the rest of the nonlinear term as a coefficient. For example,
the Fisher-KPP equation:
\[
u_{t}=u_{xx}+u(1-u)
\]
we use an implicit scheme to get:
\[
U_{m}^{n+1}-U_{m}^{n}=r(U_{m+1}^{n+1}-2U_{m}^{n+1}+U_{m}^{n+1})+\dt\cdot U_{m}^{n+1}(1-U_{m}^{n+1})
\]
Then by freezing the nonlinearity we get:
\[
U_{m}^{n+1}-U_{m}^{n}=r(U_{m+1}^{n+1}-2U_{m}^{n+1}+U_{m}^{n+1})+\dt\cdot U_{m}^{n+1}(1-\tilde{U})
\]
\\
Now that the nonlinearity is frozen, it is treated as some constant
coefficient, and we estimate it by using a worst-case bounding for
the solution. Determining the use of the maximum or minimum of the
matrix $U$ depends on which makes the magnitude of the error factor
greatest in the spectral analysis, Nonlinear freezing may be easier
computationally, but by only bounding the solution, there is chance
that the solution could contain stable oscillations.\\
Performing stability analysis requires the same techniques as used
on a linear PDE plus an analysis of the error created to linearize
the PDE. In this way we define stability conditions, oscillation free
conditions and balanced eigenvalue conditions which we can compare
to those used on linear problems.

Theoretical analysis provides us with a base which we can use to launch
our numerical analysis from. Next, we will explain how we used what
we knew theoretically, and how we applied that knowledge to create
simulations with strong and sound tests, to accurately falsify or
confirm our conjectures.

\section{Methodology}

During simulation, we wished to check three main components: Is the
solution stable? Do positive eigenvalues dominate the matrix? And
do oscillations appear in the solution? Creating numerical systems
to check these conditions challenged us to use definitions for each
of these components and implement them. Using the conditions provided
by the theoretical analysis, we checked to see if the numerical tests
measured up to each of these conditions. In our first steps, we numerically
verified that the methods were consistent and usable with Linear PDEs.
Using explicit, implicit, and semi-implicit numerical methods, we
ran simulations on each numerical scheme for each equation By beginning
with Linear PDEs we could develop sound methods which we could apply
to more complicated nonlinear equations, after a linearization was
applied.

\subsection{Stability}

Using Burden and Faire's definition of stability (\ref{Burden Faires Stability}),
a Von Neumann analysis appropriately prescribed a condition for each
scheme. Numerically tracking stability required us to cap the growth
of the solution. Burden and Faires used the standard $l_{2}$ norm,
but considering the computational strain it required, we sought for
a better fit. We settled on the infinity norm (also known as a supremum
norm) which is defined by the magnitude of the largest component of
the matrix\label{Infinity Norm Definition}. We will prove that $||U^{n+1}||_{\infty}\leq||M^{n+1}||_{2}||U_{0}||_{2}$,
which means that the infinity norm is a tighter condition than an
$l_{2}$ norm, making it a necessary condition for stability.
\begin{lem}
\textup{$||U^{n+1}||_{\infty}\leq||M^{n+1}||_{2}||U_{0}||_{2}$}\end{lem}
\begin{proof}
We will prove this lemma by simplifying the right hand side. For a
two-level difference scheme (Def. 1),
\[
||M^{n+1}||_{2}||U_{0}||_{2}\geq||M^{n+1}U_{0}||_{2}=||U^{n+1}||_{2}
\]
By definition of the 2-norm,
\[
||U^{n+1}||_{2}^{2}=\sum_{m=1}^{m}(U_{m}^{n+1})^{2}\geq(U_{k}^{n+1})^{2}\ \forall k
\]
 which implies that 
\[
\sum_{m=1}^{m}(U_{m}^{n+1})^{2}\geq\max_{k}(U_{k}^{n+1})^{2},
\]
which in turn implies that 
\[
||U^{n+1}||_{2}=\sqrt{\sum_{m=1}^{m}(U_{m}^{n+1})^{2}}\geq\sqrt{\max_{k}(U_{m}^{n+1})^{2}}=\max_{k}|U_{k}^{n+1}|=||U^{n+1}||_{\infty}.
\]
Thus, 
\[
||U^{n+1}||_{\infty}\leq||M^{n+1}||_{2}||U_{0}||_{2},
\]
so the infinity norm provides a tighter bound (\ref{Infinity Norm Definition})
\end{proof}
Realizing that when M is powered up over each temporal step, we prescribed
a numerical threshold on the growth of the matrix, U. At first, it
seemed appropriate to track the norm of the solution vector U, but
M's growth implies the growth of U, i.e. 
\[
||U^{n+1}||_{\infty}\geq K||U_{0}||_{2}\mbox{ for }K\gg1\implies||M^{n+1}||\geq K\gg1
\]
Tracking the matrix multiplier this way, allows us to cap the solution
before it gets too big because we know that it is unstable.

\subsection{Numerical Oscillations}

Oscillations are very hard to track numerically. It requires knowing
the behavior of a solution and then finding a manner to track things
different than the prescribed behavior. The theoretical monotonicity
condition ($\ref{Monotonicity Definition}$) helped us develop a numerical
monotonicity check. This definition allowed us to detect oscillations
in several initial boundary value problems. A dispersion and dissipation
analysis, allowed us to determine the PDE's natural monotonic behavior
(if it behaved monotonically), and thus we could prescribe a monotonicity
test which would most appropriately detect oscillations. From our
theoretical definition we developed the following numerical test:
\begin{equation}
sign(U_{m}^{n+2}-U_{m}^{n+1})*sign(U_{m}^{n+1}-U_{m}^{n})\geq0\label{Numerical Oscillation Test}
\end{equation}
By comparing the signs of differences between two previous temporal
or spatial steps (depending on the behavior of the solution), we could
determine if any oscillations occur by a ``flip-flopping'' of signs
across an equilibrium position.\\
Our numerical test face some limitations with other PDEs. The Fisher-KPP
failed the monotonicity check since, with most initial conditions,
there is a diffusive stage before the wave can propagate. We are unable
to differentiate physical waves from numerical oscillations by analyzing
the solution itself, but have found that eigenvalues of the method
can indicate such behavior well.

\subsection{Eigenvalues}

Preliminary investigation through simulations showed a correlation
between positive dominant eigenvalues (a single large positive eigenvalue
dominating the behavior of the matrix) and oscillation-free solutions.
Since the eigenvalues of the matrix directly relate to the growth
factor (\ref{Growth Factor-Eigenvalue relation}), we can conclude
that if all of the eigenvalues are positive and have magnitude less
than or equal to 1, we know that the solution will be stable and oscillation-free.
This provides a lower bound on the oscillation-free condition, which
we denote by calculating when $Re(\lambda_{i})>0.$ Qualitatively,
we noticed as well that there was a positive and negative balance
between eigenvectors for simulations of our Heat FTCS scheme. If a
negative eigenvalue had a corresponding positive eigenvalue with approximately
equal magnitude, their eigenvectors also correspond in magnitude.
This lead us to believe that there existed an oscillation free condition
when the eigenvalues must be balanced, or positively dominant, which
we track numerically by calculating $\max(Re(\lambda_{i}))-\max(|Re(\lambda_{i})|)\geq0$.
We can calculate if there exists an eigenvalue greater in magnitude
(but negative) than the maximum.

\subsection{Equations for Analysis}

We focused on bounded PDEs which presented Initial Boundary Value
Problems (IBVPs), using Neumann and Dirichlet boundary conditions.
We began running simulations with linear PDEs to create a foundation
for nonlinear PDEs. Two second order spatial PDEs were chosen--Heat
Equation and the Linear Reaction-Diffusion Equation. After initial
results and analysis were mostly completed on linear PDEs we selected
a couple nonlinear equations for analysis: Nonlinear Reaction-Diffusion
and Fisher-KPP Equation. Each of these has a variety of nonlinear
terms which would allow us to find patterns in the effectiveness of
the linearizations.\\
We chose very simply three basic schemes for numerical analysis of
the linear PDEs: an explicit scheme with Forward Euler, an implicit
scheme using Backward Euler, and the time averaging Crank-Nicolson.
These schemes are all very simple and relatively inexpensive to calculate.
The Crank-Nicolson provides the advantage that, through the time averaging,
\\
To tackle the nonlinear equations, we used similar schemes, but replaced
Crank-Nicolson with a Semi-Implicit scheme. To effectively calculate
the solution, we employed two forms of linearization. First, we used
a Nonlinear Freezing technique. This breaks up the nonlinear terms
into a linear term with the nonlinear factor treated as a constant
coefficient, or frozen in time. This requires matrix updates for the
estimations and can be very expensive computationally. For the second,
we used a linear approximation method (\ref{Linear Approximation})which
we can use to create an estimation for the nonlinear term. This is
a computationally simple task, that tends to be very accurate, and
promised to be the best estimator for each of the non-linear terms.

\subsection{Program Algorithm}

We used MATLAB for our analysis. Our primary program ran simulations
for a range of $\triangle x$ and $\triangle t$ values such that
$\dt,\dx\in[0.01,1]$. Nested for-loops stepped through each successive
$\triangle x$ value and $\triangle t$ value. For each simulation,
our program numerically tested for stability, positively dominated
eigenvalues and oscillations (each test is described in their sections),
storing each result in a binary matrix. Each equation with its respective
scheme had a separate MATLAB file, which contained its matrix (or
linearized matrix for the non-linear problems) and their theoretical
conditions. After collecting this information and running the simulations,
the primary analyzer compiles a visual output, color coded by which
conditions were met--from red, meaning unstable, to dark blue, for
oscillation-free (see Figure 4.1). The visual also graphed the theoretical
and analytical condition curves, which allowed us to compare the numerical
data with analysis data. Using another loop, we had the analyzer run
through every single scheme file we produced for each of the equation's
numerical schemes.

\section{Results}

While our main analyzer function provides the majority of the experimental
data, we used other simulation programs to visualize the solution
located in a specific region on the main analyzer's visual output.
If we required further investigation, we simply exported the $\triangle x$
and $\triangle t$ values and ran the simulation for the scheme. The
numerical evidence verified our conjectures. Below is a summary table
for each of the equations and schemes we investigated and their results.

\subsection*{Heat Equation ($U_{t}=U_{xx}$)}

The Heat Equation was our most basic equation that we focused on,
and most of the results from the other equations vary predictably
from the initial results presented here.We analyzed the equation and
ran simulations using both Neumann and Dirichlet boundary conditions
under the three main schemes. The explicit scheme for the Heat Equation
produces the most quintessential picture. Each region beneath the
stability curves is filled with a color corresponding to each of the
expected conditions with a solid blue beneath the monotonicity condition
and a small amount of unstable oscillations (marked by yellow) above
it, turning into complete instability (marked by red) above. Using
these simple linear PDEs we were able to test some of our numerical
checks to see if they aligned with our analysis and their conditions.\\
The Heat Equation's implicit scheme produced little of interest. Our
numerical analysis predicted unconditional stability, and picture
was consistent to our prediction, by a completely dark blue window.
However, the Crank-Nicolson method produced a very intriguing picture.

The linear front (as seen in Figure 4.2), where stable oscillations
began to form puzzled us. Considering that all of the analytical conditions
were quadratic in nature due to the numerical schemes, we struggled
to find an explanation for the linear front presented by the numerical
analysis, since for every other file our tests performed suitably.
Additionally, none of the analysis predicted the existence of stable
oscillations forming in this particular scheme. Further efforts to
devise a possible explanation should be made in the future.

\subsection*{Linear Reaction-Diffusion Equation ($U_{t}=U_{xx}-U$)}

The Linear Reaction Diffusion differs only slightly from the Heat
Equation. A reaction term makes the solution collapse to the end state
much quicker, but not much else is changed otherwise. The analyzer
program provided the exact same results as the Heat Equation. Reaction-Diffusion
FTCS produced a beautiful spread of color each underneath their expected
conditions, with the BTCS complete unstable and curious regions of
oscillations occurring behind a linear front on the Crank-Nicolson
scheme.

\subsection*{Nonlinear Reaction-Diffusion Equation ($U_{t}=U_{xx}-U^{2}$)}

Under a stable and consistent linearization technique, the Nonlinear
Reaction-Diffusion Equation behaves exactly like its linear cousin.
Under a purely explicit scheme, the equation has nearly identical
stability behavior to the linear version. Our analysis showed that
with just a numerical estimation of the Non-Linear term in an explicit
scheme, reasonable step-sizes can produce accurate and oscillation-free
results. Using a Semi-Implicit scheme produced curious results. While
stability was assured, oscillations still crept into the solution
making this method unreliable.

\subsection*{Nonlinear Fisher-KPP Equation ($U_{t}=U_{xx}+(1-U)\cdot U$)}

The Fisher Equation provides the greatest variance from the initial
linear model of the Heat Equation. The Fisher Equation provided difficulties
in tracking monotonicity and oscillations, considering the nature
of the solution containing physical waves in 2D as well as diffusive
then expansive behavior in single dimensional space. But with manual
investigation, we found that there were stable oscillations that occurred
between the balanced eigenvalue condition and the loose Von Neumann
condition. \\
\\
Below is a summary table of each of the equations with the results
from each scheme included.

\begin{table}[H]
\begin{centering}
\textsc{\Large{}Possible Stability and Oscillat}\textsc{\large{}ory
Behaviors}\\
\textsc{\large{}for given equation and scheme}\\
\textsc{\large{}\smallskip{}
}
\par\end{centering}{\large \par}

\begin{centering}
Key: U - Unstable, SO - Stable Oscillations, OFS - Oscillation Free
Stable
\par\end{centering}

\centering{}%
\begin{tabular}{|c|c|c|c|}
\hline 
Linear Equation & FTCS & CN & BTCS\tabularnewline
\hline 
\hline 
Heat/Diffusion & OFS, U & OFS, SO & OFS\tabularnewline
\hline 
Linear Reaction Diffusion & OFS, SO, U & OFS & OFS\tabularnewline
\hline 
\end{tabular}\\
\medskip{}
\begin{tabular}{|c|c|c|c|c|}
\hline 
Nonlinear Equation & FTCS & Semi-Implicit & BTCS w/ Freeze & BTCS w/ LinApprox\tabularnewline
\hline 
\hline 
Nonlinear Reaction Diffusion & OFS, SO, U & OFS & OFS & OFS\tabularnewline
\hline 
Fisher-KPP & OFS, U & OFS & OFS & OFS\tabularnewline
\hline 
\end{tabular}
\end{table}

The figures for our files are compared by scheme and represent, the
Heat Equation, Nonlinear Reaction Diffusion Equation, and the Fisher-KPP
Equation, respectively.

\begin{figure}[H]
\caption{\protect\includegraphics[scale=0.5]{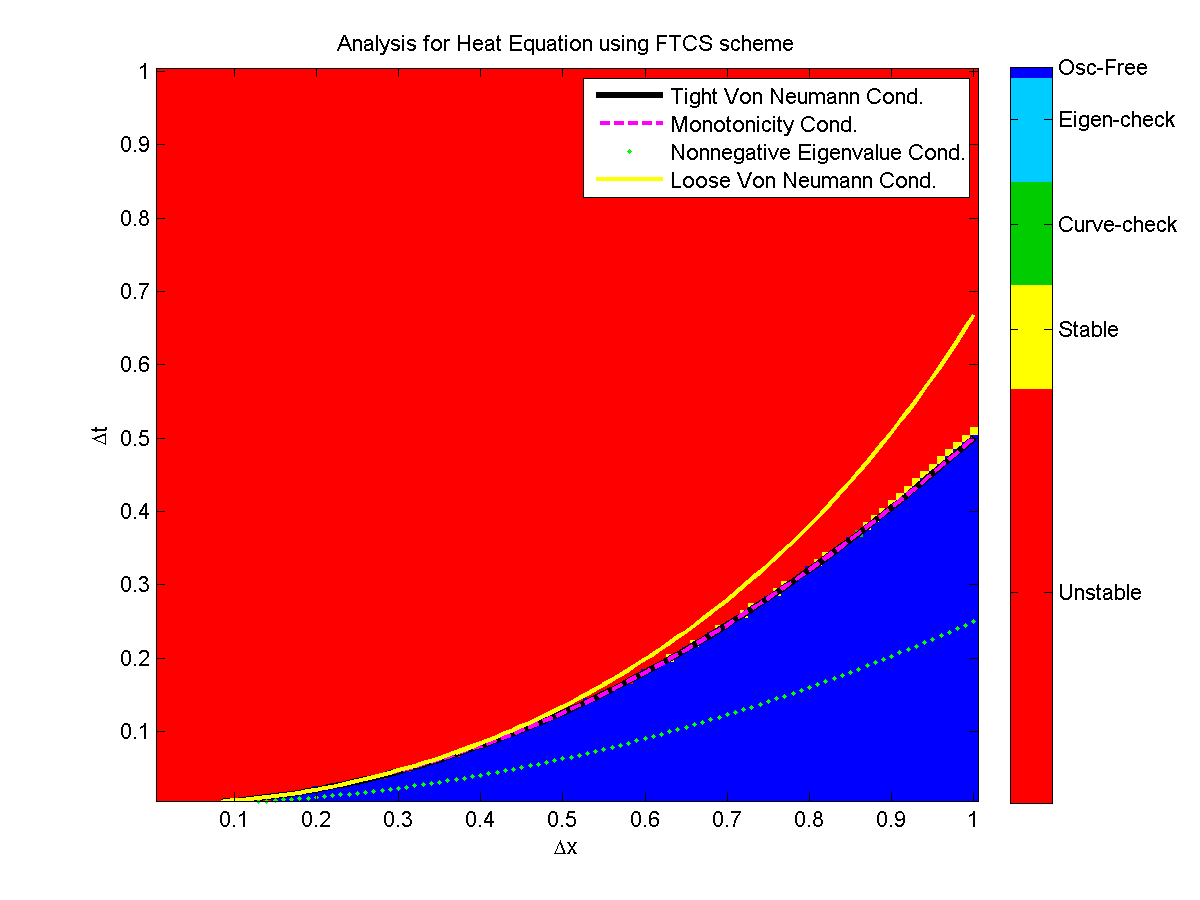}}
\caption{\protect\includegraphics[scale=0.5]{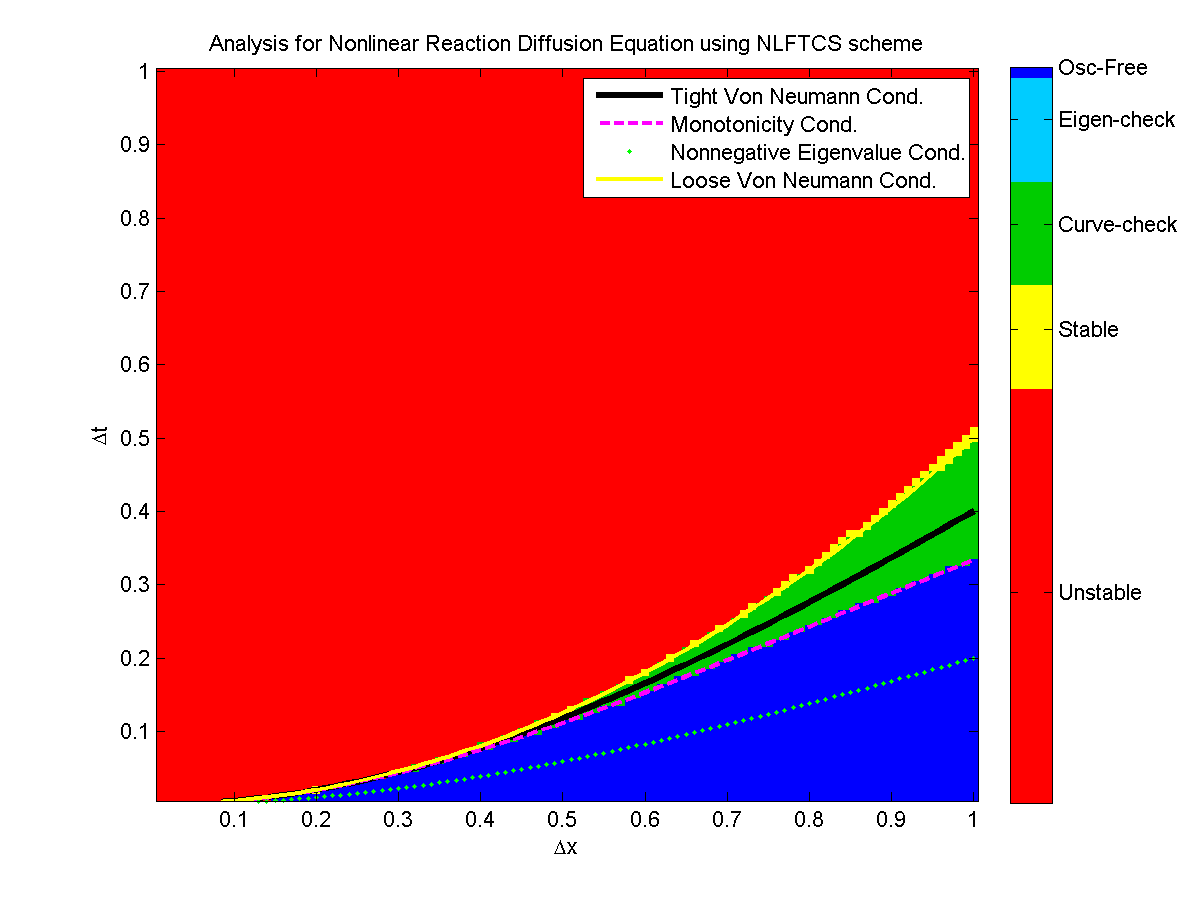}}

\caption{\protect\includegraphics[scale=0.5]{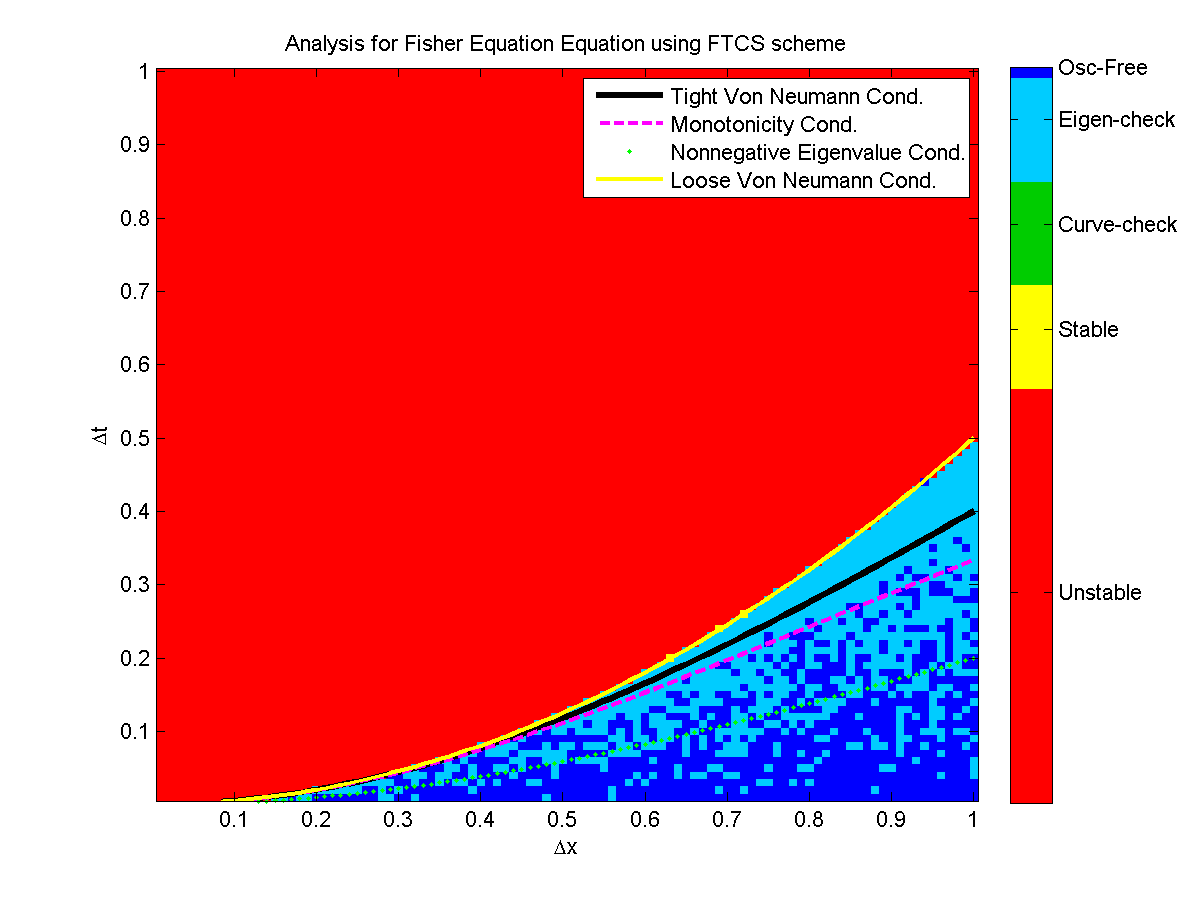}}
\end{figure}

\begin{figure}[H]
\caption{\protect\includegraphics[scale=0.5]{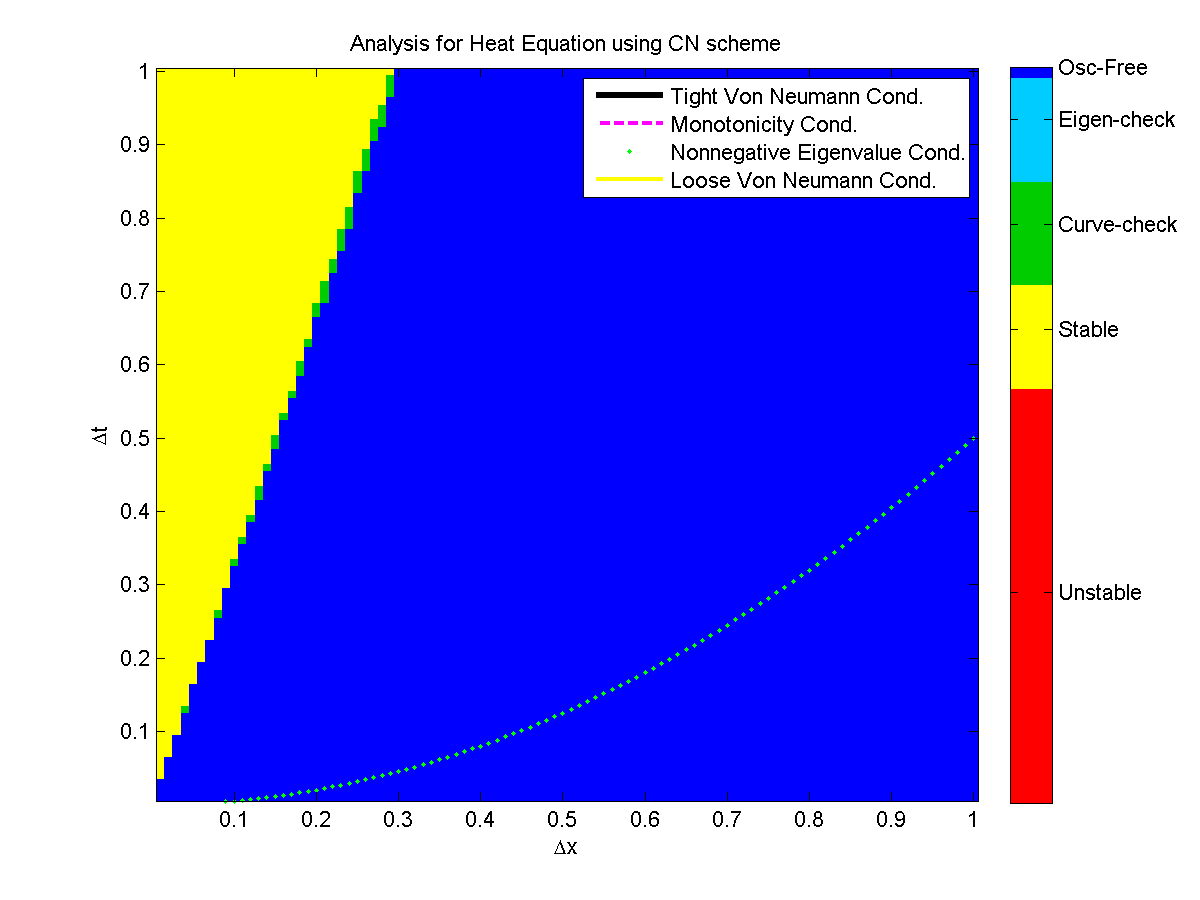}}
\caption{\protect\includegraphics[scale=0.5]{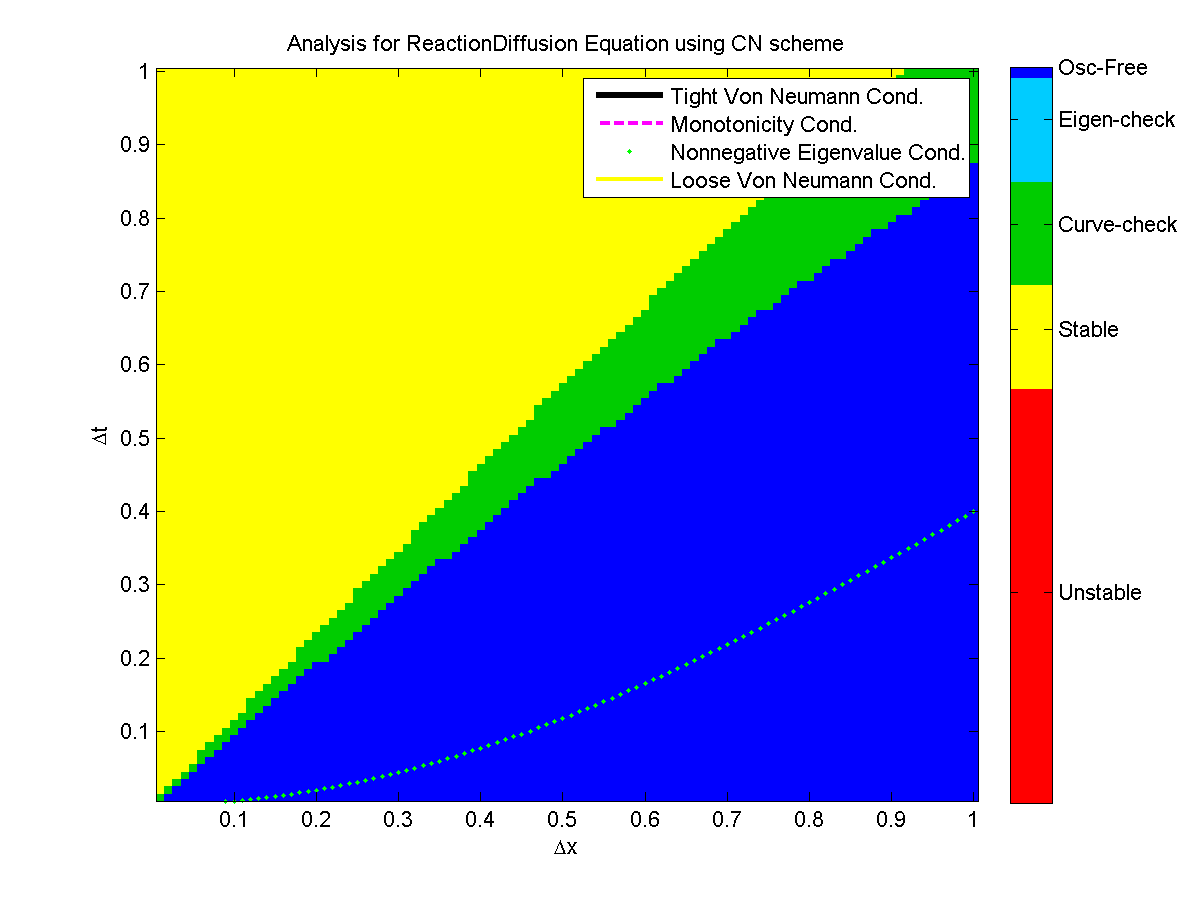}}

\caption{\protect\includegraphics[scale=0.5]{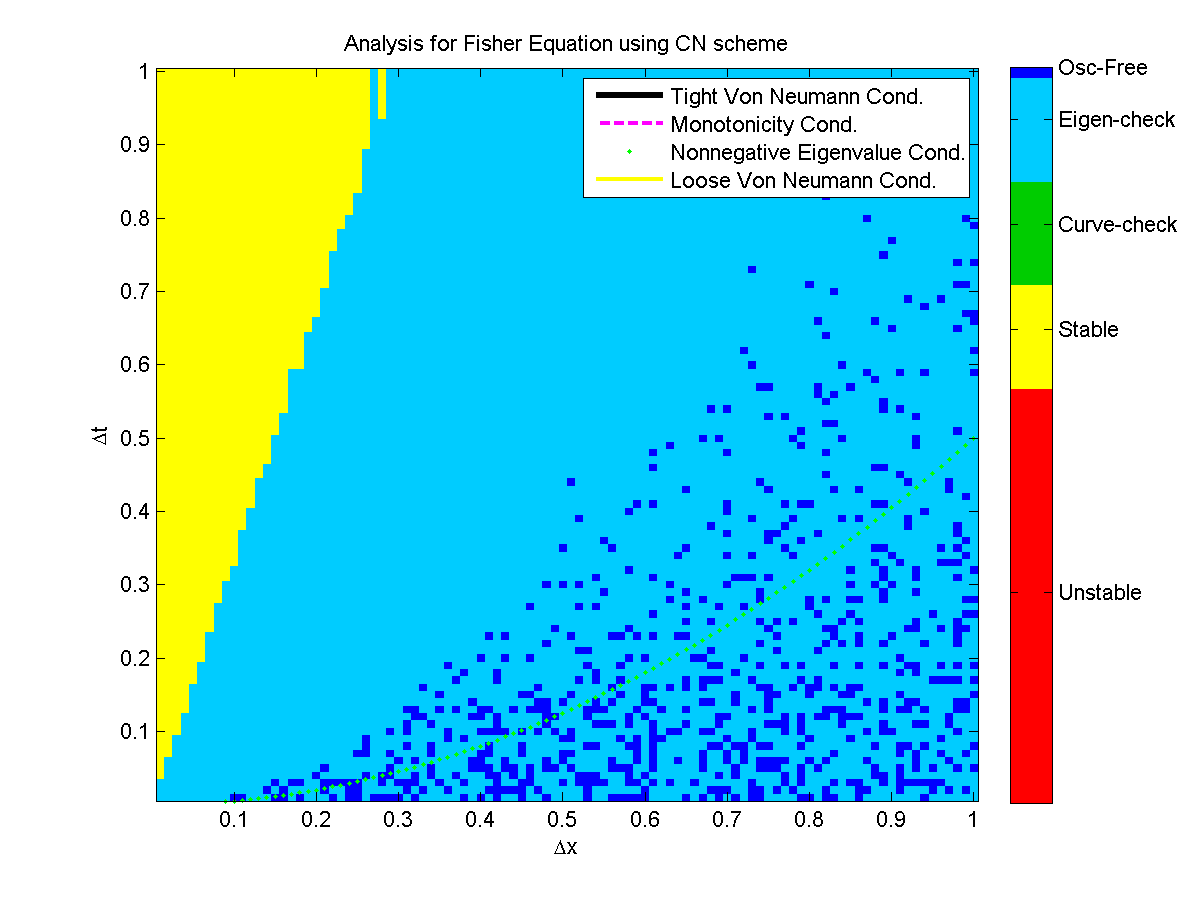}}
\end{figure}

\section{Discussion}

The diagnosis of numerical oscillations is important, since eliminating
them can provide accurate representations of these descriptive and
powerful equations, which are incredibly useful to us. The eigenvalues
of linear problems provide a concrete connection to the stability
and oscillations conditions of numerical schemes, and with a spectral
analysis, we can force a condition on the real parts of the eigenvalues
remaining positive. More usefully however is the connection between
eigenvalues and the Von Neumann error factor, especially in semi-linear
or non-linear PDEs. This error factor is a powerful tool in estimating
the eigenvalues of these nonlinear PDEs. The following conjectures,
backed up by numerical evidence, demonstrate that a consistent linearization
of a problem may allow us to form conditions on the linearized scheme
using the error factor. Thus, we would be able to contain these oscillations
for nonlinear parabolic PDEs as well.
\begin{conjecture}
Given a linear PDE with a symmetric positive difference scheme matrix,
if there exists a region of numerically stable oscillations, it is
bounded below by the positive eigenvalue condition
\end{conjecture}
From preliminary investigation, a relation between the balance of
eigenvalues and the monotonicity condition seems evident, most particularly
from our data from the Heat FTCS analysis (see Figure 4.1). In analyzing
the eigenvalues and eigenvectors, we noticed that if there were paired
eigenvalues of similar magnitude by opposite sign, then the eigenvectors
were of equal magnitude but opposite sign. This led us to a believe
that the balance of these eigenvalues along with their eigenvectors
had a place in an oscillation free condition. The balanced eigenvalues
and eigenvectors seemed to ``block'' oscillations.
\begin{conjecture}
Given a linear PDE with a well posed scheme and consistent scheme
matrix, the solution will be oscillation-free iff there exists a dominating
positive eigenvalue.
\end{conjecture}
The numerical evidence for this is very convincing, since the monotonicity
condition relates directly to this balanced eigenvalue condition.
We have seen this in our numerical results, which is summarized in
the Results section along with the visual representation of our data.
\begin{conjecture}
Given a nonlinear partial differential equation, the stability and
oscillation-free conditions for a consistent linearized form of a
nonlinear numerical method is sufficient to ensure oscillation-free
stability.
\end{conjecture}
Tying this to the monotonicity condition leads us to the final conjecture.
We believe that if a non-linear PDE is monotonic, then it is also
oscillation-free. Since the monotonicity condition is a necessary
condition on the eigenvalues of a method as well as the fact that
it is sufficient for stability, then we believe that the solution
will be oscillation-free and stable.
\begin{conjecture}
Given a nonlinear partial differential equation with monotonic initial
and asymptotic behavior, a numerical method is oscillation-free if
it is monotonic. Further, this monotonicity condition matches the
oscillation-free condition of the linearized form of the numerical
method.
\end{conjecture}
Containing oscillations in Initial Value Boundary Problems, which
contain no physical oscillations, remained relatively easy, but exploration
of those containing physical waves and oscillations, accurately damping
numerical oscillation requires more refined investigation.

\end{document}